\documentclass[12pt]{amsart}
\usepackage{amssymb,amsmath,amscd,mathtools}
\usepackage[margin=1in]{geometry}

\theoremstyle{plain}
\newtheorem{theorem}{Theorem}[section]

\newtheorem*{quest}{Main Question}
\newtheorem{corollary}[theorem]{Corollary}
\newtheorem{obs}[theorem]{Observation}

\theoremstyle{definition}

\numberwithin{equation}{section}

\renewcommand{\div}{\operatorname{div_{\mathbb{H}}}}

\newcommand{\eucl}{\operatorname{eucl}}

\newcommand{\ipg}[2]{\langle #1,#2 \rangle_\mathbb{G}}
\newcommand{\iph}[2]{\langle #1,#2 \rangle_\mathbb{H}}
\newcommand{\tp}{\texttt{p}}
\renewcommand{\span}{\operatorname{span}}
\newcommand{\divergence}{\operatorname{div}}
\def\abs#1{\left|#1\right|}
\def\norm#1{\left\|#1\right\|}
\newcommand{\defeq}{\vcentcolon=}

\begin{document}

\title[Generalizations of the Drift Laplace Equation]{Generalizations of the Drift Laplace Equation in the Heisenberg Group and a Class of Grushin-Type Spaces}

\author{Thomas Bieske}
\author{Keller Blackwell}
\address{Department of Mathematics\\
University of South Florida\\ 
Tampa, FL 33620, USA}
\email{tbieske@mail.usf.edu}
\address{Department of Mathematics\\
University of South Florida\\ 
Tampa, FL 33620, USA}
\email{kellerb@mail.usf.edu}

\date{May 29,2019}

\begin{abstract}
    We find fundamental solutions to $\tp$-Laplace equations with drift terms in the Heisenberg group and Grushin-type planes. These solutions are natural generalizations to the fundamental solutions discovered by Beals, Gaveau, and Greiner for the Laplace equation with drift term. Our results are independent of the results of Bieske and Childers, in that Bieske and Childers consider a generalization that focuses on a $\tp$-Laplace-type equation while we primarily concentrate on a generalization of the drift term. 
\end{abstract}
\maketitle

\section{Introduction and Motivation}
In \cite{BGG}, Beals, Gaveau, and Greiner establish a formula for the fundamental solution to the Laplace equation with drift term in a large class of sub-Riemannian spaces. (See Sections 2 and 3 for definitions and further discussion.) In \cite{B:C}, Bieske and Childers expanded these results by invoking a $\tp$-Laplace generalization that encompasses the formulas of \cite{BGG} and discovered a negative result \cite[Theorems 4.1, 4.2]{B:C}.  In this paper, we focus on that negative result and produce a natural generalization of the $\tp$-Laplace equation with drift term. Our solutions are stable under limits when $\tp\to\infty$ and when the drift parameter $L\to 0$ (which is the standard $\tp$-Laplace equation). 

This paper is the result of an undergraduate research project by the second author under the direction of the first. The second author would like to thank the University of South Florida Honors College and the Department of Mathematics and Statistics for their support and research opportunities. 

\section{The Environments}
We concern ourselves with two sub-Riemannian environments, the Heisenberg group and Grushin-type planes, which are 2-dimensional Grushin-type spaces. We will recall the construction of these spaces and then highlight their main properties.

\subsection{The Heisenberg Group}
We begin with $\mathbb{R}^3$ using the coordinates $(x_1,x_2,x_3)$ and consider the linearly
independent vector fields $ \{X_1, X_2, X_3\}$, defined by:
\begin{equation*}
 X_1 = \dfrac{\partial}{\partial x_1} - \frac{x_2}{2}\dfrac{\partial}{ \partial x_3}, \; X_2 = \dfrac{\partial}{\partial x_2} + \frac{x_1}{2}\dfrac{\partial}{\partial x_3},\; \text{and } X_3 = \dfrac{\partial}{\partial x_3}
\end{equation*} 
which obey the relation
\begin{equation*}
 	[X_1,X_2] = X_3. 
\end{equation*}
We then have a Lie Algebra denoted $h_1$ that decomposes as a direct sum 
$h_1 = V_1 \oplus V_2$
where $V_1 = \span\lbrace X_1, X_2 \rbrace$ and $V_2 = \span\lbrace X_3 \rbrace$. The Lie algebra is statified; i.e., $[V_1, V_1] = V_2$ and $[V_1, V_2] = 0$.
We endow $h_1$ with an inner
product $\iph{\cdot \,}{\cdot}$ and related norm $\|\cdot\|_{\mathbb{H}}$ so that this basis is orthonormal. 

The corresponding Lie Group is called the general Heisenberg group of dimension $1$ and is denoted by $\mathbb{H}^1$. With this choice of vector fields the exponential map is the identity map, so that for any $p,q$ in $\mathbb{H}^1$, written as $p=(x_1, x_2, x_3)$ and $q=(\widehat{x_1}, \widehat{x_2}, \widehat{x_3})$ the group multiplication law is given by 
\begin{equation*}
p \cdot q = \left( x_1+\widehat{x_1}, x_2+\widehat{x_2}, x_{3}+ \widehat{x_3} + \frac{1}{2}
(x_1\widehat{x_2} - x_{2} \widehat{x_1}) \right).
\end{equation*}

The natural metric on $\mathbb{H}^1$ is the Carnot-Carath\'{e}odory metric  given by 
$$d_C(p,q)= \inf_{\Gamma} \int_{0}^{1} \| \gamma '(t) \|_{\mathbb{H}}\; dt $$
where the set $\Gamma$ is the set of all curves $ \gamma $ such that $\gamma (0) = p, \gamma (1) = q$ and $\gamma'(t) \in V_1$.  By Chow's theorem (See, for example, \cite{BR:SRG}.) any two points can be connected by such a curve, which makes $d_C(p,q)$ a left-invariant metric on $\mathbb{H}^1$.

Given a smooth function $u:\mathbb{H}^1 \to \mathbb{R}$, we define the horizontal gradient by
$$\nabla_0 u = (X_1u,X_2u).$$

Additionally, given a vector field $F=\sum_{i=1}^{2}f_iX_i+f_{3}X_3$, we define the Heisenberg divergence of $F$, denoted $\div F$, 
by $$\div F = \sum_{i=1}^{2} X_if_i.$$ 

A quick calculation shows that when $f_{3}=0$, we have $$\div F = \textmd{div}_{\eucl}F$$ where 
$\textmd{div}_{\eucl}$ is the standard Euclidean divergence. The main operator we are concerned with is the 
horizontal $\tp$-Laplacian for
$1< \tp <\infty$ defined by \begin{eqnarray}
\Delta_\tp u & = & \div(\|\nabla_{0} u\|_{\mathbb{H}}^{\tp-2}\nabla_{0}u)  =  \sum_{i=1}^{2}X_i\big(\|\nabla_{0} u\|_{\mathbb{H}}^{\tp-2}X_iu\big) \nonumber\\
& = & \frac{\tp-2}{2}\|\nabla_{0} u\|_{\mathbb{H}}^{\tp-4}\sum_{i=1}^{2}X_i\|\nabla_{0} u\|_{\mathbb{H}}^{2}X_iu+ \|\nabla_{0} u\|_{\mathbb{H}}^{\tp-2}\sum_{i=1}^{2}X_iX_iu.   \label{heisenbergplap}
\end{eqnarray}

For a more complete treatment of the Heisenberg group, the interested
reader is directed to \cite{BR:SRG}, \cite{B:HG}, \cite{F:SE}, \cite{FS:HSHG}
\cite{G:MS}, \cite{H:CCG}, \cite{K:LGHT}, \cite{St:HA} and the references therein.

\subsection{Grushin-type planes}
The Grushin-type planes differ from the Heisenberg group in that Grushin-type planes lack an algebraic group law. We begin with $\mathbb{R}^{2}$, possessing coordinates $(y_1, y_2)$, $a\in \mathbb{R}$, $c\in \mathbb{R}\setminus\{0\}$ and $n\in \mathbb{N}$.  We use them to construct the vector fields:
\begin{equation*}
Y_1  =  \frac{\partial}{\partial y_1} \ \textmd{and}\ \
Y_2  =   c(y_1-a)^n\frac{\partial}{\partial y_2}.
\end{equation*}
For these vector fields, the only (possibly) nonzero Lie bracket is
\begin{equation*}
[Y_1,Y_2]=cn(y_1-a)^{n-1}\frac{\partial}{\partial y_2}.
\end{equation*}
Because $n\in \mathbb{N}$, it follows that H\"{o}rmander's condition is satisfied by these vector fields.

We will put a (singular) inner product on $\mathbb{R}^{2}$, denoted $\ipg{\cdot}{\cdot}$,  with related norm $\|\cdot\|_{\mathbb{G}}$, so that the collection $\{Y_{1}, Y_{2}\}$ 
forms an orthonormal basis. We then have a sub-Riemannian space that we will call $g_{n}$, which is also the tangent space to a generalized Grushin-type plane $\mathbb{G}_n$. Points in $\mathbb{G}_n$ will also be denoted by
$p=(y_{1}, y_{2})$.  The Carnot-Carath\'{e}odory distance on $\mathbb{G}_n$ is defined for points $p$ and $q$ as follows:
\begin{eqnarray*}
d_{\mathbb{G}}(p,q)=\inf_{\Gamma}\int \|\gamma'(t)\|_{\mathbb{G}}\;dt
\end{eqnarray*} 
with $\Gamma$ the set of curves $\gamma$ such that $\gamma(0)=p$, $\gamma(1)=q$ and
$
\gamma'(t)\in \span\{Y_{1}(\gamma(t)),Y_{2}(\gamma(t))\}.
$
By Chow's theorem, this is an honest metric.

We shall now discuss calculus on the Grushin-type planes. Given a smooth function $f$ on $\mathbb{G}_n$, we define the horizontal gradient of $f$ as
\begin{equation*}
 \nabla_{0} f(p) = \big(Y_1f(p),Y_2f(p)\big).
\end{equation*}

Using these derivatives, we consider a key operator on $C^2_{\mathbb{G}}$ functions, namely the \tp-Laplacian for $1<\tp<\infty$, given by
\begin{eqnarray}
\Delta_\tp f & = & \divergence_\mathbb{G} (\|\nabla_{0} f\|_{\mathbb{G}}^{\tp-2}\nabla_{0}f)   =  Y_1\big(\|\nabla_{0} f\|_{\mathbb{G}}^{\tp-2}Y_1f\big)+Y_2\big(\|\nabla_{0} f\|_{\mathbb{G}}^{\tp-2}Y_2f\big)  \nonumber\\
& = & \frac{1}{2}(\tp-2)\|\nabla_{0} f\|_{\mathbb{G}}^{\tp-4}\big(Y_1\|\nabla_{0} f\|_{\mathbb{G}}^{2}Y_1f+Y_2\|\nabla_{0} f\|_{\mathbb{G}}^{2}Y_2f\big)  \label{grushinplap}\\
& & \mbox{} +\|\nabla_{0} f\|_{\mathbb{G}}^{\tp-2}\big(Y_1Y_1f+Y_2Y_2f\big).  \nonumber
\end{eqnarray}

\section{Motivating Results}

\subsection{The Heisenberg Group}
Capogna, Danielli, and Garofalo \cite{CDG} proved the following theorem.
\begin{theorem}[\cite{CDG}]\label{Heisenbergplap}
Let $1<\tp<\infty$. In $\mathbb{H}^1$, let $$u(x_1,x_2,x_3)=\left( x_1^2 + x_2^2 \right)^2+16 x_3^2.$$ For $\tp \neq 4$, let $$\eta_{\tp}=\frac{4-\tp}{4(1-\tp)},$$ and let
\begin{eqnarray*}
\zeta_{\tp}=\left\{\begin{array}{cc}
u(x_1,x_2,x_3)^{\eta_{\tp}} & \tp \neq 2n+2 \\
\log u(x_1,x_2,x_3) & \tp = 2n+2.
\end{array}\right.
\end{eqnarray*}
Then we have $\Delta_{\tp}\zeta_{\tp}= C\delta_0$ for some constant $C$ in the sense of distributions.
\end{theorem}

In the Heisenberg Group, Beals, Gaveau, and Greiner \cite{BGG} extend this equation as shown in the following theorem (cf. \cite[Theorem 3.4]{B:C}). 

\begin{theorem}[\cite{BGG}]\label{HeisenbergD2lap}
Let $L \in \mathbb{R}$, $\abs{L} \ne 1$. Consider the following constants,
\begin{equation*}
\eta = \frac{L-1}{2} \ \ \textmd{and}\ \ \tau  =  \frac{-(L+1)}{2} 
\end{equation*}
together with the functions,
\begin{eqnarray*}
v(x_1,x_2,x_3)  =  \left( x_1^2 + x_2^2 \right) - 4 i x_3 & \textmd{and} &
w(x_1,x_2,x_3)  =  \left( x_1^1 + x_2^2 \right) + 4 i x_3
\end{eqnarray*}
to define our main function, $u_{2,L}(x_1,x_2,x_3)$ given by
$$u_{2,L}(x_1,x_2,x_3)= v(x_1,x_2,x_3)^\eta w(x_1,x_2,x_3)^\tau.$$ 
Then we have $\Delta_{2}u_{2,L}+iL[X_1,X_2]u_{2,L}= C\delta_0$ for some constant $C$ in the sense of distributions.
\end{theorem}

\subsection{Grushin-type Planes}
Bieske and Gong \cite{BG} proved the following in the Grushin-type planes.
\begin{theorem}[\cite{BG}]\label{Grushinplap}
Let $1<\tp<\infty$ and define $$f(y_1,y_2)=c^2(y_1-a)^{(2n+2)}+(n+1)^2 (y_2-b)^2.$$  
For $\tp \neq n+2$, consider $$\tau_{\tp}=\frac{n+2-\tp}{(2n+2)(1-\tp)}$$
so that in $\mathbb{G}_n$  we have the well-defined function
\begin{eqnarray*}
\psi_{\tp}=\left\{\begin{array}{cc}
f(y_1,y_2)^{\tau_{\tp}} & \tp \neq n+2 \\
\log f(y_1,y_2) & \tp = n+2.
\end{array}\right.
\end{eqnarray*}
Then, $\Delta_{\tp}\psi_{\tp}= C\delta_0$ for some constant $C$ in the sense of distributions.
\end{theorem} 
In the Grushin-type planes, Beals, Gaveau and Greiner \cite{BGG} extend this equation as shown in the following theorem (cf. \cite[Theorem 3.2]{B:C}). 
\begin{theorem}[\cite{BGG}]\label{GrushinD2lap}
Let $L \in \mathbb{R}$, $\abs{L} \ne 1$. Consider the following quantities,
\begin{equation*}
\alpha = \frac{-n}{(2n+2)}(1+L) \ \ \textmd{and}\ \ \beta  =  \frac{-n}{(2n+2)}(1-L).
\end{equation*}
We use these constants with the functions
\begin{eqnarray*}
g(y_1,y_2) & = & c(y_1-a)^{n+1}+i(n+1)(y_2-b)\\ 
h(y_1,y_2) & = & c(y_1-a)^{n+1}-i(n+1)(y_2-b)
\end{eqnarray*}
to define our main function $f_{2,L}(y_1,y_2)$, given by
\begin{eqnarray*}
f_{2,L}(y_1,y_2) & = & g(y_1,y_2)^{\alpha}h(y_1,y_2)^{\beta}.
\end{eqnarray*}
Then, $\Delta_{2}f_{2,L}+iL[Y_1,Y_2]f_{2,L}=C\delta_0$ for some constant $C$ in the sense of distributions.
\end{theorem}

\begin{obs}
We have the following well-known observations from \cite{B:C}. In $\mathbb{H}^1\setminus\{0\}$, $$\zeta_2 (x_1, x_2, x_3)= \left( \left(x_1^2 + x_2^2)^2 \right)^2 + 16 x_3^3 \right)^{-\frac{1}{2}}$$ solves $$\Delta_2 \zeta_2=0.$$
Also, $$ u_{2,L}(x_1,x_2,x_3)=v(x_1,x_2,x_3)^{\frac{L-1}{2}}w(x_1,x_2,x_3)^{-\frac{(L+1)}{2}}$$
solves $$\Delta_2 u_{2,L} +iL[X_1,X_2]u_{2,L}=0.$$
The equations and solutions coincide when $L=0$; i.e., $u_{2,0} = \zeta_2$. Similarly, in $\mathbb{G}_n\setminus\{(a,b)\}$, we have when $\tp=2$, $$\psi_2(y_1,y_2)=\big(c^2(y_1-a)^{2n+2}+(n+1)^2(y_2-b)^2\big)^{\displaystyle{-\frac{n}{2n+2}}}$$ solves $$\Delta_2 \psi_2=0.$$
Also, $$f_{2,L} (y_1,y_2)=g(y_1,y_2)^{-\frac{n}{2n+2}(1+L)}h(y_1,y_2)^{-\frac{n}{2n+2}(1-L)}$$ 
solves $$\Delta_2 f_{2,L} +iL[Y_1,Y_2] f_{2,L}=0.$$
Notice that the equations and solutions coincide when $L=0$; i.e., $f_{2,0}=\psi_2.$
\end{obs}

\begin{quest}
We wish to extend the preceding relationship in $\mathbb{G}_n\setminus\{(a,b)\}$ and in $\mathbb{H}^1\setminus\{0\}$ for all $\tp,\ 1<\tp < \infty$. In the case of the Grushin-type planes, we wish to find a differential operator $\mathcal{G}_{\tp, L}$ and a function $f_{\tp, L}$ satisfying:
\begin{equation*}
    \mathcal{G}_{\tp, 0} = \Delta_\tp \quad \text{and} \quad \mathcal{G}_{2,L} = \Delta_2 + iL[Y_1, Y_2]
\end{equation*}
with $f_{\tp, 0}$ being the solution of Theorem \ref{Grushinplap} and $f_{2,L}$ being the solution of Theorem \ref{GrushinD2lap} such that:
\begin{equation*}
    \mathcal{G}_{\tp, L} f_{\tp, L}(q) = 0
\end{equation*}
for $q \in \mathbb{G}_n \setminus \lbrace (a,b) \rbrace$, $1 < \tp < \infty$, and $L \in \mathbb{R}$. Similarly, in the case of the Heisenberg group, we wish to find a differential operator $\mathcal{H}_{\tp, L}$ and a function $u_{\tp, L}$ satisfying:
\begin{equation*}
    \mathcal{H}_{\tp, 0} = \Delta_\tp \quad \text{and} \quad \mathcal{H}_{2,L} = \Delta_2 + iL[X_1, X_2]
\end{equation*}
with $u_{\tp, 0}$ being the solution of Theorem \ref{Heisenbergplap} and $u_{2,L}$ being the solution of Theorem \ref{HeisenbergD2lap} such that:
\begin{equation*}
    \mathcal{H}_{\tp, L} u_{\tp, L}(q) = 0
\end{equation*}
for $q \in \mathbb{H}^1 \setminus \lbrace 0 \rbrace$, $1 < \tp < \infty$, and $L \in \mathbb{R}$.

Furthermore, we would like $f_{\tp,L}$ and $u_{\tp,L}$ to be the fundamental solutions to their respective equations. 
\end{quest}

\section{A Generalization in the Heisenberg Group}

For the Heisenberg group $\mathbb{H}^1$, we consider the following parameters:
\begin{eqnarray*}
    \eta = \frac{4-\tp + 2L(1-\tp)}{4(1-\tp)} & \textmd{and} & \tau = \frac{4-\tp - 2L(1-\tp)}{4(1-\tp)}
\end{eqnarray*}
for $L \in \mathbb{R}$ with:
\begin{eqnarray*}
L \ne \pm\frac{4 - \tp}{2(1-\tp)}.
\end{eqnarray*}
We use these parameters with the functions
\begin{eqnarray*}
v(x_1,x_2,x_3) & = & (x_1^2 + x_2^2) - 4i x_3\\ 
w(x_1,x_2,x_3) & = & (x_1^2 + x_2^2) + 4i x_3
\end{eqnarray*}
to define our main function:
\begin{eqnarray}
\label{coreheisenberg}
u_{\tp, L}(y_1, y_2) = v(x_1,x_2,x_3)^\eta w(x_1,x_2,x_3)^\tau.
\end{eqnarray}
Using Equation \ref{coreheisenberg}, we have the following theorem.

\begin{theorem}
On $\mathbb{H}^1$, we have:
\begin{eqnarray*}
\mathcal{H}_{\tp, L} \left( u_{\tp, L} \right) \defeq \Delta_{\tp} u_{\tp, L} + iL \left[ X_1, X_2 \right] \left( \norm{\nabla_0 u_{\tp, L}}^{\tp -2}_\mathbb{H} u_{\tp, L}  \right) = C\delta_0
\end{eqnarray*}
for some constant $C$ in the sense of distributions.
\end{theorem}

\begin{proof}
Suppressing arguments and subscripts, we compute the following:
\begin{eqnarray}
X_1 u & = & 2 v^{\eta-1} w^{\tau -1} \big( (\eta w + \tau v)x_1 + (\eta w - \tau v) i x_2 \big) \label{X1} \\ \nonumber
\overline{X_1 u} & = & 2 v^{\tau-1} w^{\eta -1} \big( (\eta v + \tau w) x_1 + (\eta v - \tau w) i x_2 \big)\\
X_2 u & = & 2 v^{\eta -1} w^{\tau -1} \big( (\eta w + \tau v) x_2 - (\eta w - \tau v) i x_1 \big) \label{X2} \\ \nonumber
\overline{X_2 u} & = & 2 v^{\tau -1} w^{\eta -1} \big( (\eta v + \tau w) x_2 - (\eta v - \tau w) i x_1 \big)\\
\textmd{and so \ }\norm{\nabla_0 u}^2 & = & 8(\eta^2 + \tau^2) v^{\eta + \tau -1} w^{\eta + \tau -1} (x_1^2 + x_2^2). \label{Hnormsq}
\end{eqnarray}
Using the above we compute
\begin{eqnarray*}
X_1(X_1 u) & = & 2 v^{\eta -2} w^{\tau -2} \Big( 2 \big( (\eta w + \tau v)x_1^2 + (-\eta w - \tau v) i x_1 x_2 \big) \big( (\eta - 1)w + (\tau -1)v \big)\\
&& \mbox{}+ 2i \big( (\eta w + \tau v) x_2^2 + (\eta w - \tau v) i x_2^2 \big) \big(-(\eta-1)w + (\tau -1)v) \big)\\
&& \mbox{} + vw \big( 2(x_1^2 + x_2^2)(\tau + \eta) + (\eta w + \tau v) \big) \Big)\\
\textmd{and \ }X_2(X_2 u) & = & 2 v^{\eta -2} w^{\tau -2} \Big( 2 \big( (\eta w + \tau v)x_2^2 + (-\eta w + \tau v) i x_1 x_2 \big) \big( (\eta - 1)w + (\tau -1)v \big)\\
&&\mbox{} +  2i \big( (\eta w + \tau v) x_1 x_2 + (-\eta w + \tau v) i x_1^2 \big) \big(-(\eta-1)w + (\tau -1)v) \big)\\
&&\mbox{} + vw \big( 2(x_1^2 + x_2^2)(\tau + \eta) + (\eta w + \tau v) \big) \Big).
\end{eqnarray*}
In addition, we have
\begin{eqnarray}
X_1 \norm{\nabla_0 u}^2 & = & 16(\eta^2 + \tau^2) v^{\eta+\tau -2} w^{\eta + \tau -2} \label{X1Hnormsq}\\ \nonumber
&& \mbox{} \times \Big(  vw x_1 + 2(\eta + \tau -1)(x_1^2 + x_2^2)^2 \big( x_1  - 4 x_2 x_3 \big)  \Big)\\
\textmd{and \ }X_2 \norm{\nabla_0 u}^2 & = & 16(\eta^2 + \tau^2) v^{\eta+\tau -2} w^{\eta + \tau -2} \label{X2Hnormsq}\\ \nonumber
&& \mbox{} \times \Big(  vw x_2 + 2(\eta + \tau -1)(x_1^2 + x_2^2)^2 \big( x_2  - 4 x_1 x_3 \big)  \Big)
\end{eqnarray}
so that
\begin{eqnarray*}
\sum_{j=1}^2 X_j \|\nabla_0 u\|^2 (X_j u) &=& 32 (\eta^2 + \tau^2) v^{2 \eta + \tau - 3} w^{\eta + 2 \tau -  3}  \Big((\eta w + \tau v)vw(x_1^2 +x_2^2)\\ 
&& \mbox{} +  2 (\eta + \tau - 1)(x_1^2 + x_2^2)^2\\
&& \mbox{} \times \Big( (\eta w+\tau v)(x_1^2 + x_2^2)^2  -4(\eta w- \tau v)ix_3 \Big) \Big)\\
\textmd{and \ }\|\nabla_0u\|^2\big(X_1X_1u+X_2X_2u\big) &=& 16 (\eta^2 + \tau^2) v^{2 \eta + \tau - 3} w^{\eta + 2 \tau - 3}(x_1^2 + x_2^2)\\
&& \mbox{} \times \Big(2vw(\eta w + \tau v)+ 4vw(\eta + \tau) (x_1^2 + x_2^2)\\
&& \mbox{} +  2 \big((\eta - 1) w + (\tau - 1) v\big)(\eta w + \tau v) (x_1^2 + x_2^2) \\
&&\mbox{}+ 2 \big(-(\eta - 1) w + (\tau - 1) v\big) (\eta w - \tau v) (x_1^2 + x_2^2)\Big).
\end{eqnarray*}
This yields
\begin{eqnarray*}
\Delta_\tp u & = & \norm{\nabla_0 u}^{\tp-4} \left( \frac{(\tp-2)}{2} \sum_{j=1}^2 X_j \|\nabla_0 u\|^2 (X_j u) + \norm{\nabla_0 u}^2 (X_1 X_1 f + X_2 X_2 u) \right)\\
& = & 2L \frac{(4-\tp)^{\tp-2}}{ (1-\tp)^{\tp-2}} \left( 1 + \frac{4L^2 (1-\tp)^2}{(4-\tp)^2} \right)^\frac{\tp-2}{2}   v^{\frac{1}{2} \left( \tp\eta + (\tp-2)\tau - \tp  \right)} w^{\frac{1}{2} \left( (\tp-2)\eta + \tp\tau - \tp  \right)} \\
&& \mbox{} \times (x_1^2 + x_2^2)^\frac{\tp-2}{2} \left( -2L (x_1^2 + x_2^2) + \tp 4i x_3 \right).
\end{eqnarray*}
We then compute
\begin{eqnarray*}
iL [X_1, X_2] \left( \norm{\nabla_0 u}^{\tp-2} u \right) &=& iL \frac{(4-\tp)^{\tp-2}}{ (1-\tp)^{\tp-2}} \left( 1 + \frac{4L^2 (1-\tp)^2}{(4-\tp)^2} \right)^\frac{\tp-2}{2} (x_1^2 + x_2^2)^\frac{\tp-2}{2}\\
&&\mbox{} \times  \dfrac{\partial}{\partial x_3} v^{\frac{1}{2}(\tp-2)(\eta + \tau - 1) + \eta} w^{\frac{1}{2}(\tp-2)(\eta + \tau - 1) + \tau}\\
&= & -2L \frac{(4-\tp)^{\tp-2}}{ (1-\tp)^{\tp-2}} \left( 1 + \frac{4L^2 (1-\tp)^2}{(4-\tp)^2} \right)^\frac{\tp-2}{2}    (x_1^2 + x_2^2)^\frac{\tp-2}{2} \\
&& \mbox{} \times  v^{\frac{1}{2} \left( \tp\eta + (\tp-2)\tau - \tp  \right)} w^{\frac{1}{2} \left( (\tp-2)\eta + \tp\tau - \tp  \right)} \left( -2L (x_1^2 + x_2^2) + \tp 4i x_3 \right)\\
&=& - \Delta_\tp u 
\end{eqnarray*}
from which it follows that $\mathcal{H}_{\tp, L} u_{\tp, L} = 0$ on $\mathbb{H}^1 \setminus \lbrace 0 \rbrace$, away from the singularity. We now consider the normalization:
\begin{eqnarray*}
v_\varepsilon (x_1,x_2,x_3) & \defeq & (x_1^2 + x_2^2) + \varepsilon^2 - 4i x_3\\ 
w_\varepsilon (x_1,x_2,x_3) & \defeq & (x_1^2 + x_2^2) + \varepsilon^2 +  4i x_3
\end{eqnarray*}
so that
\begin{eqnarray*}
u_\varepsilon (x_1, x_2, x_3) & \defeq &  v_\varepsilon (x_1,x_2,x_3)^\eta  w_\varepsilon (x_1,x_2,x_3)^\tau .
\end{eqnarray*}
Suppressing arguments and computing similarly as before yields the distribution:
\begin{eqnarray}\label{heisenbergdist}
\mathcal{H}_{\tp, L} u_\varepsilon & = & 2^\frac{3p-2}{2} \varepsilon^2 \left( \frac{p(4-p)}{4(1-p)}  +  L^2 \right) (\eta^2 + \tau^2)^{\frac{p-2}{2}} (x_1^2 + x_2^2)^{\frac{p-2}{2}}\\ \nonumber
&&\qquad \times v_\varepsilon^{\frac{\eta p + \tau (p-2) - p}{2}} w_\varepsilon^{\frac{\eta (p-2) + \tau p - p}{2}} .
\end{eqnarray}
By the argument of \cite[Theorem 7.5, (c)]{BGG}, the distribution of \eqref{heisenbergdist} is determined by the following density:
\begin{eqnarray}\label{heisenbergdensity}
\qquad \frac{2^\frac{3p-2}{2} \Big( \frac{p(4-p)}{4(1-p)}  +  L^2 \Big) (\eta^2 + \tau^2)^{\frac{p-2}{2}} \Big( \left( \frac{x_1}{\varepsilon} \right)^2 + \left( \frac{x_2}{\varepsilon} \right)^2 \Big)^{\frac{p-2}{2}} dm\Big( \frac{x_1^2 + x_2^2}{\varepsilon^2} \Big) d\Big( \frac{x_3}{\varepsilon^2} \Big) \frac{1}{-2i}}{\left(  \left( \frac{x_1}{\varepsilon} \right)^2 + \left( \frac{x_2}{\varepsilon} \right)^2 + 1 - 4i \frac{x_3}{\varepsilon^2}\right)^{-\frac{\eta p + \tau (p-2) - p}{2}} \left( \left( \frac{x_1}{\varepsilon} \right)^2 + \left( \frac{x_2}{\varepsilon} \right)^2  + 1 + 4i \frac{x_3}{\varepsilon^2}\right)^{-\frac{\eta (p-2) + \tau p - p}{2}}}
\end{eqnarray}
where $dm$ denotes the Lebesgue measure in the complex plane. Then as $\varepsilon \to 0$ the distribution of \eqref{heisenbergdensity} tends to the $\delta_0$ distribution, up to a constant factor.
\end{proof}
Observing that:
\begin{eqnarray*}
    L \ne \pm \frac{4-\tp}{2(1-\tp)} & \textmd{implies} & \tp \ne \abs{\frac{2L + 4}{2L +1}}_,  \abs{\frac{2L - 4}{2L - 1}}
\end{eqnarray*}
we have immediately the following corollary.
\begin{corollary}\label{hsmooth}
Let $\tp > \max \left\lbrace  \abs{\frac{2L + 4}{2L + 1}}_,  \abs{\frac{2L - 4}{2L - 1}}  \right\rbrace$. Then the function $u_{\tp, L}$ of Equation \ref{coreheisenberg} is a smooth solution to the Dirichlet problem
\begin{eqnarray*}
\left\{\begin{array}{cc}
\mathcal{H}_{tp, L} \left( u_{\tp, L}(q) \right) =0 & q \in \mathbb{H}^1\setminus\{0\} \\
0 & q = 0.
\end{array}\right.
\end{eqnarray*}
\end{corollary}

\section{A Generalization in the Grushin Plane}

For the Grushin-type planes, we consider the following parameters:
\begin{eqnarray*}
    \alpha = \frac{n+2-\tp - Ln(1- \tp)}{2(n+1)(1-\tp)} & \textmd{and} & \beta = \frac{ n+2-\tp + Ln(1- \tp)}{2(n+1)(1-\tp)}
\end{eqnarray*}
where $L \in \mathbb{R}$ with:
\begin{eqnarray*}
L \ne \pm \frac{n+2-\tp}{n(1-\tp)}.
\end{eqnarray*}
We use these constants with the functions
\begin{eqnarray*}
g(y_1,y_2) & = & c(y_1-a)^{n+1}+i(n+1)(y_2-b)\\ 
h(y_1,y_2) & = & c(y_1-a)^{n+1}-i(n+1)(y_2-b)
\end{eqnarray*}
to define our main function:
\begin{eqnarray}
\label{coregrushin}
f_{\tp, L}(y_1, y_2) = g(y_1, y_2)^\alpha h(y_1, y_2)^\beta.
\end{eqnarray}
Using Equation \ref{coregrushin}, we have the following theorem.

\begin{theorem}
On $\mathbb{G}_n$, we have:
\begin{eqnarray*}
\mathcal{G}_{\tp, L} \left( f_{\tp, L} \right) \defeq \Delta_{\tp} f_{\tp, L} + iL \left[ Y_1, Y_2 \right] \left( \norm{\nabla_0 f_{\tp, L}}^{\tp -2}_\mathbb{G} f_{\tp, L}  \right) = C\delta_0 
\end{eqnarray*}
for some constant $C$ the sense of distributions.
\end{theorem}

\begin{proof}
Suppressing arguments and subscripts, we compute the following:
\begin{eqnarray}
Y_1 f & = & c(n+1)(y_1 - a)^n g^{\alpha -1} h^{\beta -1} (\alpha h + \beta g)\\ \label{Y1}  \nonumber
\overline{Y_1 f} & = & c(n+1)(y_1 - a)^n g^{\beta -1} h^{\alpha -1} (\alpha g + \beta h)\\ 
Y_2 f & = & ic(n+1) (y_1 - a)^n g^{\alpha -1} h^{\beta -1} (\alpha h - \beta g)\\ \label{Y2} \nonumber
\overline{Y_2 f} & = & ic(n+1) (y_1 - a)^n g^{\beta -1} h^{\alpha -1} (\alpha g - \beta h)\\ 
\textmd{and so \ }\norm{\nabla_0 f}^2 & = & 2c^2 (n+1)^2 (y_1 - a)^{2n} g^{\alpha + \beta -1} h^{\alpha + \beta -1} (\alpha^2 + \beta^2). \label{Gnormsq}
\end{eqnarray}
Using the above we compute:
\begin{eqnarray}
\nonumber Y_1 (Y_1 f) & = & c(n+1) (y_1 - a)^{n-1} g^{\alpha - 2} h^{\beta -2}  \\
&& \nonumber \mbox{} \times \Big( ngh(\alpha h + \beta g) + c(n+1)(y_1 - a)^{n+1}  \\
&& \nonumber \mbox{} \times  \big( (\alpha h + \beta g)\left( (\alpha - 1)h + (\beta - 1)g \right) + gh(\alpha + \beta) \big) \Big)\\
\nonumber Y_2 (Y_2 f) & = & -c^2 (n+1)^2 (y_1 - a)^{2n} g^{\alpha -2} h^{\beta -2}  \times\\ \nonumber
& & \big( (\alpha h - \beta g) \left( (\alpha - 1)h - (\beta - 1)g \right) - gh(\alpha + \beta) \big) \\
Y_1 \norm{\nabla_0 f}^2 & = & 4c^2 (n+1)^2 (\alpha^2 + \beta^2) (y_1 - a)^{2n-1} g^{\alpha + \beta -2} h^{\alpha + \beta -2} \label{Y1Gnormsq}\\ \nonumber
& & \mbox{} \times \big( ngh + c^2 (n+1) (\alpha + \beta -1)(y_1 - a)^{2n+2} \big)\\
Y_2 \norm{\nabla_0 f}^2 & = & 4c^3 (n+1)^4 (\alpha^2 + \beta^2) (y_1 - a)^{3n} (y_2 - b) \label{Y2Gnormsq}\\
&&\nonumber  \times (\alpha + \beta -1) g^{\alpha + \beta -2} h^{\alpha + \beta -2}
\end{eqnarray}
and
\begin{eqnarray*}
\sum_{i=1}^2 Y_i\|\nabla_0 f\|^2 (Y_i f) & = & 
4c^3(n+1)^3(\alpha^2+\beta^2)(y_1-a)^{3n-1}g^{2\alpha+\beta-3}h^{\alpha+2\beta-3}\\
&&\mbox{}  \times\Big((\alpha h+\beta g)\big(ngh+c^2(n+1)(\alpha+\beta-1)(y_1-a)^{2n+2}\big)\\
&&\mbox{} + ic(n+1)^2(y_1-a)^{n+1}(y_2-b)(\alpha+\beta-1)(\alpha h -\beta g)\Big)\\
\|\nabla_0f\|^2(Y_1Y_1f+Y_2Y_2f) & = & 2c^3(n+1)^3(\alpha^2+\beta^2)(y_1-a)^{3n-1}g^{2\alpha+\beta-3}h^{\alpha+2\beta-3}\\
&&\mbox{}\times\Big(ngh(\alpha h+\beta g)+4c(n+1)(y_1-a)^{n+1}gh(\alpha\beta)\Big)
\end{eqnarray*}
so that
\begin{eqnarray*}
\Delta_\tp f &=& \norm{\nabla_0 f}^{\tp-4} \left( \frac{(\tp-2)}{2} \sum_{j=1}^2 Y_j\|\nabla_0 f\|^2 (Y_j f) + \norm{\nabla_0 f}^2 (Y_1 Y_1 f + Y_2 Y_2 f) \right)\\
& = & -L 2^{\frac{\tp-2}{2}} c^{\tp-1} n^2 (n+1)^{\tp-2} (y_1 - a)^{n(\tp-1)-1} (\alpha^2 + \beta^2)^{\frac{\tp-2}{2}}\\
&& \mbox{} \times g^{\frac{1}{2}\left( \alpha \tp + \beta(\tp-2) - \tp  \right)} h^{\frac{1}{2}\left( \alpha (\tp-2) + \beta \tp - \tp  \right)} \left( L c(y_1-a)^{n+1} + i(1-\tp)(n+1)(y_2 - b) \right).
\end{eqnarray*}
We then compute:
\begin{eqnarray*}
iL[Y_1, Y_2] \left( \norm{\nabla_0 f}^{\tp -2} f \right) & = & iL 2^{\frac{\tp-2}{2}} c^{\tp-1} n (n+1)^{\tp-2} (y_1 - a)^{n(\tp-1)-1} (\alpha^2 + \beta^2)^{\frac{\tp-2}{2}}  \\
& & \mbox{}  \times \dfrac{\partial}{\partial y_2} \left(  g^{\frac{1}{2} (\alpha \tp + \beta(\tp-2) - (\tp-2))} h^{\frac{1}{2}(\alpha(\tp-2) + \beta \tp - (\tp-2))} \right)\\
& = & L 2^{\frac{\tp-2}{2}} c^{\tp-1} n^2 (n+1)^{\tp-2} (y_1 - a)^{n(\tp-1)-1} (\alpha^2 + \beta^2)^{\frac{\tp-2}{2}}\\
&& \mbox{} \times g^{\frac{1}{2}(\alpha \tp + \beta(\tp-2) - \tp)} h^{\frac{1}{2}(\alpha(\tp-2) + \beta \tp - \tp)}\\
&& \mbox{} \times  \left( Lc(y_1-a)^{n+1}  + i (1-\tp)(n+1)(y_2 - b) \right)\\
&= & - \Delta_\tp f
\end{eqnarray*}
from which it follows that $\mathcal{G}_{\tp, L} f_{\tp, L} = 0$ on $\mathbb{G}_n \setminus \lbrace (a,b) \rbrace$, away from the singularity. We now consider the normalization:
\begin{eqnarray*}
g_\varepsilon (y_1, y_2) & \defeq & c(y_1 - a)^n + \varepsilon^2 + i(n+1)(y_2-b)\\
h_\varepsilon (y_1, y_2) & \defeq & c(y_1 - a)^n + \varepsilon^2 - i(n+1)(y_2-b)
\end{eqnarray*}
so that:
\begin{eqnarray*}
f_\varepsilon (y_1, y_2) & \defeq & g_\varepsilon (y_1, y_2)^\alpha h_\varepsilon (y_1, y_2)^\beta .
\end{eqnarray*}
Suppressing arguments and computing similarly as before yields the distribution:
\begin{eqnarray} \label{grushindist}
\mathcal{G}_{\tp, L} f_\varepsilon &=& -2^{\frac{p-2}{2}} \varepsilon^2 \left( (n + 2 - p) - nL^2 \right)   c^{p-1} n (n+1)^{p-2}  \left( \alpha^2  + \beta^2 \right)^{\frac{p-2}{2}}\\ \nonumber
&&\qquad \times (y_1 - a)^{n(p-1)-1} g^{\frac{\alpha p + \beta(p-2) - p}{2}} h^{\frac{\alpha(p-2) + \beta p - p}{2}}.
\end{eqnarray}
By the argument of \cite[Theorem 7.5, (c)]{BGG}, the distribution of \eqref{grushindist} is determined by the following density:
\begin{eqnarray}\label{grushindensity}
&& -2^{\frac{p-2}{2}}  \big( (n + 2 - p) - nL^2 \big)   c^{p-1} n (n+1)^{p-2}\left( \alpha^2  + \beta^2 \right)^{\frac{p-2}{2}}\\ \nonumber
&&\qquad \times \left( \frac{y_1 - a}{\varepsilon^{2/(n+1)}} \right)^{n(p-1)-1} dm\left( \frac{y_1 - a}{\varepsilon^{2/(n+1)}} \right) d\left( \frac{y_2 - b}{\varepsilon^2} \right) \left( \frac{1}{-2i} \right)\\ \nonumber
&&\qquad \times \left( c \left(\frac{y_1 - a}{\varepsilon^{2/(n+1)}} \right)^{n+1} + 1 + i(n+1) \frac{(y_2 -b)}{\varepsilon^2} \right)^{\frac{\alpha p + \beta(p-2) - p}{2}}\\ \nonumber
&&\qquad \times \left( c \left(\frac{y_1 - a}{\varepsilon^{2/(n+1)}} \right)^{n+1} + 1 - i(n+1) \frac{(y_2 -b)}{\varepsilon^2} \right)^{\frac{\alpha(p-2) + \beta p - p}{2}}
\end{eqnarray}
where $dm$ denotes the Lebesgue measure in the complex plane. Then as $\varepsilon \to 0$ the distribution of \eqref{grushindensity} tends to the $\delta_0$ distribution, up to a constant factor.
\end{proof}
Observing that
\begin{eqnarray*}
    L \ne \pm \frac{n(\tp -1)}{n+2 - \tp} & \textmd{implies} & \tp \ne \abs{\frac{L(n+2) + n}{n+ L}}_,  \abs{\frac{L(n+2) - n}{n- L}}
\end{eqnarray*}
we have immediately the following corollary.
\begin{corollary}\label{gsmooth}
Let $\tp > \max\left\lbrace \abs{\frac{L(n+2) + n}{n+ L}}_,  \abs{\frac{L(n+2) - n}{n- L}} \right\rbrace $. Then the function $f_{\tp, L}$ of Equation \ref{coregrushin} is a smooth solution to the Dirichlet problem
\begin{eqnarray*}
\left\{\begin{array}{cc}
\mathcal{G}_{\tp, L} \left( f_{\tp, L}(q) \right) =0 & q \in \mathbb{G}_n\setminus\{(a,b)\} \\
0 & q = (a,b).
\end{array}\right.
\end{eqnarray*}
\end{corollary}


\section{The Limit as $\tp \to \infty$}

\subsection{Heisenberg Group}

Recall that the drift $\tp$-Laplace equation in the Heisenberg group $\mathbb{H}^1$ is given by:
\begin{equation*}
\mathcal{H}_{\tp,L}(u) \defeq \Delta_\tp u + iL [X_1, X_2] \left( \norm{\nabla_0 u}_\mathbb{H}^{\tp-2} u \right) = 0.
\end{equation*}
A routine expansion of the drift term yields the observation:
\begin{eqnarray*}
\mathcal{H}_{\tp,L}(u) &=& \Delta_\tp u + iL\left( \frac{\tp - 2}{2} \norm{\nabla_0 u}_\mathbb{H}^{\tp-4} \left( \dfrac{\partial}{\partial x_3} \norm{\nabla_0 u}_\mathbb{H}^{2} \right)u + \norm{\nabla_0 u}_\mathbb{H}^{\tp-2} \dfrac{\partial}{\partial x_3}u \right)\\
&=& 0.
\end{eqnarray*}
Dividing through by $\frac{\tp - 2}{2} \norm{\nabla_0 u}_\mathbb{H}^{\tp-4}$ and formally taking the limit $\tp \to \infty$, we obtain:
\begin{equation*}
\mathcal{H}_{\infty, L}(u) = \Delta_\infty u + iL [X_1, X_2] \left( \norm{\nabla_0 u}_\mathbb{H}^{2} \right) u.
\end{equation*}
Considering Equation \ref{coreheisenberg} and formally letting $\tp \to \infty$ yields:
\begin{equation*}
u_{\infty, L}(x_1, x_2, x_3) = v(x_1, x_2, x_3)^{\frac{1+2L}{4}} w(x_1, x_2, x_3)^{\frac{1-2L}{4}} 
\end{equation*}
where we recall the functions $v(x_1, x_2, x_3)$ and $w(x_1, x_2, x_3)$ are given by:
\begin{eqnarray*}
v(x_1, x_2, x_3) &=& \left( x_1^2 + x_2^2 \right) - 4ix_3\\
w(x_1, x_2, x_3) &=& \left( x_1^2 + x_2^2 \right) + 4ix_3.
\end{eqnarray*}
We have the following theorem.
\begin{theorem}
The function $u_{\infty,L}$, as above, is a smooth solution to the Dirichlet problem
\begin{eqnarray*}
\left\{\begin{array}{cc}
\mathcal{H}_{\infty,L} u_{\infty,L}(q)  =0 & q \in \mathbb{H}^1 \setminus\{0\} \\
0 & q = 0.
\end{array}\right.
\end{eqnarray*}
\end{theorem}

\begin{proof}
We may prove this theorem by letting $\tp\to\infty$ in Equations \eqref{X1}, \eqref{X2}, \eqref{X1Hnormsq}, \eqref{X2Hnormsq} and invoking continuity (cf. Corollary \ref{hsmooth}). However, for completeness we compute formally. We let:
\begin{eqnarray*}
N = \frac{1+2L}{4} &\textmd{and} & T = \frac{1-2L}{4}
\end{eqnarray*}
and, suppressing arguments and subscripts, compute:
\begin{eqnarray*}
X_1 u & = & 2 v^{N-1} w^{T -1} \big( (N w + T v)x_1 + (N w - T v) i x_2 \big) \\ 
X_2 u & = & 2 v^{N -1} w^{T -1} \big( (N w + T v) x_2 - (N w - T v) i x_1 \big) \\
\norm{\nabla_0 u}^2 & = & 8(N^2 + T^2) v^{N + T -1} w^{N + T -1} (x_1^2 + x_2^2)\\
X_1 \norm{\nabla_0 u}^2 & = & 16(N^2 + T^2) v^{N+T -2} w^{N + T -2}\\
&&\mbox{} \times \Big(  vw x_1 + 2(N + T -1)(x_1^2 + x_2^2)^2 \big( x_1  - 4 x_2 x_3 \big)  \Big)\\
\textmd{and\ }X_2 \norm{\nabla_0 u}^2 & = & 16(N^2 + T^2) v^{N+T -2} w^{N + T -2}\\ 
&&\mbox{} \times\Big(  vw x_2 + 2(N + T -1)(x_1^2 + x_2^2)^2 \big( x_2  - 4 x_1 x_3 \big)  \Big)
\end{eqnarray*}
so that
\begin{eqnarray*}
\Delta_\infty u &=& X_1 \norm{\nabla_0 u}^2 X_1 u + X_2 \norm{\nabla_0 u}^2 X_2 u\\
&=&   32 (N^2 + T^2) v^{2 N + T - 3} w^{N + 2 T -  3}  \Big((N w + T v)vw(x_1^2 +x_2^2)\\ 
&& \mbox{} +  2 (N + T - 1)(x_1^2 + x_2^2)^2\\
&& \mbox{} \times \Big( (N w+T v)(x_1^2 + x_2^2)^2  - 4(N w- T v)ix_3 \Big) \Big)\\
&=& 128 i L(N^2 + T^2) (x_1^2 + x_2^2 ) x_3 v^{2N + T - 2} w^{N + 2T - 2}.
\end{eqnarray*}
We also compute:
\begin{eqnarray*}
iL [X_1, X_2] \left( \norm{\nabla_0 u}^2 \right) u &=& iL v^N w^T \dfrac{\partial}{\partial x_3} \norm{\nabla_0 f}^2\\
&=& -128 i L(N^2 + T^2) (x_1^2 + x_2^2 ) x_3 v^{2N + T - 2} w^{N + 2T - 2}.
\end{eqnarray*}
The theorem follows.
\end{proof}
We notice that when $L=0$, this result was a part of the Ph.D. thesis of the first author \cite{BT}. 
In particular, combined with \cite{BT, B:HG}, we have shown the following diagram commutes in $\mathbb{H}^1\setminus\{0\}$:
$$\begin{CD}
\mathcal{H}_{\tp,L} \left( u_{\tp,L} \right)=0 @>>{\tp\to\infty}>\mathcal{H}_{\infty,L} \left(u_{\infty,L} \right)=0 \\
@VV{L\to 0}V                @VV{L\to 0}V \\
\Delta_{\tp}u_{\tp,0}=0 @>>{\tp\to\infty}> \Delta_{\infty}u_{\infty,0}=0
\end{CD}$$

\subsection{Grushin-type Planes}

Recall that the drift $\tp$-Laplace equation in the Grushin-type planes $\mathbb{G}_n$ is given by:
\begin{equation*}
\mathcal{G}_{\tp,L}(f) \defeq \Delta_\tp f + iL [Y_1, Y_2] \left( \norm{\nabla_0 f}_\mathbb{G}^{\tp-2} f \right) = 0.
\end{equation*}
A routine expansion of the drift term yields the observation
\begin{eqnarray*}
\mathcal{G}_{\tp,L}(f) &=& \Delta_\tp f + iL cn (y_1 -a)^{n-1} \left( \frac{\tp-2}{2} \norm{\nabla_0 f}^{\tp-4}_\mathbb{G} \left( \dfrac{\partial}{\partial y_2} \norm{\nabla_0 f}^{2}_\mathbb{G}  \right) f + \norm{\nabla_0 f}^{\tp-2}_\mathbb{G} \dfrac{\partial}{\partial y_2}f  \right)\\
&=& 0.
\end{eqnarray*}
Dividing through by $\frac{\tp-2}{2} \norm{\nabla_0 f}^{\tp-4}_\mathbb{G}$ and formally taking the limit $\tp \to \infty$, we obtain:
\begin{equation*}
\mathcal{G}_{\infty, L}(f) = \Delta_\infty f + iL [Y_1, Y_2] \left( \norm{\nabla_0 f}_\mathbb{G}^{2} \right) f.
\end{equation*}
Considering Equation \ref{coregrushin} and formally letting $\tp \to \infty$ yields:
\begin{equation*}
f_{\infty, L}(y_1, y_2) = g(y_1, y_2)^{\frac{1}{2(n+1)}(1-nL)} h(y_1, y_2)^{\frac{1}{2(n+1)}(1+nL)} 
\end{equation*}
where we recall the functions $g(y_1, y_2)$ and $h(y_1, y_2)$ are given by:
\begin{eqnarray*}
g(y_1, y_2) &=& c(y_1 - a)^{n+1} + i(n+1)(y_2-b)\\
h(y_1, y_2) &=& c(y_1 - a)^{n+1} - i(n+1)(y_2-b).
\end{eqnarray*}

We have the following theorem.
\begin{theorem}
The function $f_{\infty,L}$, as above, is a smooth solution to the Dirichlet problem
\begin{eqnarray*}
\left\{\begin{array}{cc}
\mathcal{G}_{\infty,L} f_{\infty,L}(q)=0 & q \in \mathbb{G}_n\setminus\{(a,b)\} \\
0 & q = (a,b).
\end{array}\right.
\end{eqnarray*}
\end{theorem}
\begin{proof}
We may prove this theorem by letting $\tp\to\infty$ in Equations \eqref{Y1},\eqref{Y2},\eqref{Y1Gnormsq},\eqref{Y2Gnormsq} and invoking continuity (cf. Corollary \ref{gsmooth}). However, for completeness we compute formally. We let:
\begin{eqnarray*}
A = \frac{1}{2(n+1)}(1-nL) & \textmd{and} & B = \frac{1}{2(n+1)}(1+nL)
\end{eqnarray*}
and, suppressing arguments and subscripts, compute:
\begin{eqnarray*}
Y_1 f & = & c(n+1)(y_1 - a)^n g^{A -1} h^{B -1} (A h + B g)\\
Y_2 f & = & ic(n+1) (y_1 - a)^n g^{A -1} h^{B -1} (A h - B g)\\
\norm{\nabla_0 f}^2 & = & 2c^2 (n+1)^2 (y_1 - a)^{2n} g^{A + B -1} h^{A + B -1} (A^2 + B^2)\\
Y_1 \norm{\nabla_0 f}^2 & = & 4c^2 (n+1)^2 (A^2 + B^2) (y_1 - a)^{2n-1} g^{A + B -2} h^{A + B -2} \\
& &\mbox{}  \times  \big( ngh + c^2 (n+1) (A + B -1)(y_1 - a)^{2n+2} \big)\\
\textmd{and\ }Y_2 \norm{\nabla_0 f}^2 & = & 4c^3 (n+1)^4 (A^2 + B^2) (y_1 - a)^{3n} (y_2 - b) \\
&&  \mbox{} \times (\alpha + \beta -1) g^{A + B -2} h^{A + B -2}
\end{eqnarray*}
so that
\begin{eqnarray*}
\Delta_\infty f &=& Y_1 \norm{\nabla_0 f}^2 Y_1 f + Y_2 \norm{\nabla_0 f}^2 Y_2 f \\
& = & 4 c^3(n+1)^3(A^2+B^2)(y_1-a)^{3n-1}g^{2A+B-3}h^{A+2B-3}\\
&&\mbox{}  \times\Big((A h+B g)\big(ngh+c^2(n+1)(A +B -1)(y_1-a)^{2n+2}\big)\\
&&\mbox{} + ic(n+1)^2(y_1-a)^{n+1}(y_2-b)(A +B-1)(A h -B g)\Big)\\
&=& 4iL c^3 (n+1)^3 n^2 (A^2 + B^2) (y_1 - a)^{3n-1} (y_2 - b) g^{2A + B - 2} h^{A + 2B -2}.
\end{eqnarray*}
We also compute:
\begin{eqnarray*}
iL[Y_1, Y_2] \left( \norm{\nabla_0 f}^2 \right) f &=& iL g^A h^B \left( cn(y_1 - a)^{n-1} \dfrac{\partial}{\partial y_2}  \norm{\nabla_0 f}^2 \right)\\
&=& -4iL c^3 (n+1)^3 n^2 (A^2 + B^2) g^{2A + B - 2} h^{A + 2B -2}\\
&& \mbox{} \times  (y_1 - a)^{3n-1} (y_2 - b).
\end{eqnarray*}
The theorem follows.

\end{proof}

In particular, combined with \cite{BG}, we have shown the following diagram commutes in \\ $\mathbb{G}_n\setminus\{(a,b)\}$:
$$\begin{CD}
\mathcal{G}_{\tp, L} f_{\tp,L}=0 @>>{\tp\to\infty}>\mathcal{G}_{\infty, L}f_{\infty,L}=0 \\
@VV{L\to 0}V                @VV{L\to 0}V \\
\Delta_{\tp}f_{\tp,0}=0 @>>{\tp\to\infty}> \Delta_{\infty}f_{\infty,0}=0
\end{CD}$$


\end{document}